\documentclass[reqno,12pt]{amsart}
\usepackage[T1]{fontenc}
\usepackage[utf8]{inputenc}
\usepackage[english]{babel}
\usepackage{amssymb,amsmath,amsthm,amsfonts,xcolor,enumerate,hyperref,comment,longtable,cleveref,mathtools}

\usepackage{times}
\usepackage{cite}
\usepackage{pdflscape}
\usepackage[mathcal]{euscript}
\usepackage{tikz}
\usepackage{cancel}
\usepackage{stmaryrd}

\usetikzlibrary{arrows}

\usepackage[a4paper,top=3cm, bottom=3cm, left=3cm, right=3cm]{geometry}

\newtheorem{theorem}{Theorem}[section]
\newtheorem{lemma}[theorem]{Lemma}
\newtheorem{proposition}[theorem]{Proposition}

\newtheorem{definition}[theorem]{Definition}

 \DeclareMathOperator{\gr}{gr}
 \DeclareMathOperator{\Ann}{Ann}
\DeclareMathOperator{\Aut}{Aut}
\DeclareMathOperator{\Orb}{Orb}
\DeclareMathOperator{\Hom}{Hom}

\begin{document}

\title[Central extensions of filiform associative algebras]
{Central extensions of filiform associative algebras}

\author{Iqboljon Karimjanov}
\address{[Iqboljon Karimjanov] Department of Matem\'aticas, Institute of Matem\'aticas, University of Santiago de Compostela, 15782, Spain.}
\email{iqboli@gmail.com}

\author{Ivan   Kaygorodov}
\address{[Ivan   Kaygorodov] CMCC, Universidade Federal do ABC. Santo Andr\'e, Brasil.}
\email{kaygorodov.ivan@gmail.com}

\author{Manuel Ladra}
\address{[Manuel Ladra] Department of Matem\'aticas, Institute of Matem\'aticas, University of Santiago de Compostela, 15782, Spain.}
\email{manuel.ladra@usc.es}
\thanks{This work was supported by Agencia Estatal de Investigaci\'on (Spain), grant MTM2016-79661-P (European FEDER support included, UE);  RFBR 17-01-00258; FAPESP 17/15437-6.}

\begin{abstract}In this paper we describe central extensions (up to isomorphism) of all complex null-filiform and filiform associative algebras.
\end{abstract}

\subjclass[2010]{16D70, 16S70}
\keywords{associative algebras, nilpotent, null-filiform, naturally graded, filiform, quasi-filiform, $2$-cocycles
central extension.}

\maketitle

\section*{Introduction}
Central extensions play an important role in quantum mechanics: one of the earlier
encounters is by means of Wigner’s theorem which states that a symmetry of a quantum
mechanical system determines an (anti-)unitary transformation of a Hilbert space.

Another area of physics where one encounters central extensions is the quantum theory
of conserved currents of a Lagrangian. These currents span an algebra which is closely
related to the so-called affine Kac–Moody algebras, which are the universal central extension
of loop algebras.

Central extensions are needed in physics, because the symmetry group of a quantized
system usually is a central extension of the classical symmetry group, and in the same way
the corresponding symmetry Lie algebra of the quantum system is, in general, a central
extension of the classical symmetry algebra. Kac–Moody algebras have been conjectured
to be a symmetry groups of a unified superstring theory. The centrally extended Lie
algebras play a dominant role in quantum field theory, particularly in conformal field
theory, string theory and in M-theory.

In the theory of Lie groups, Lie algebras and their representations, a Lie algebra extension
is an enlargement of a given Lie algebra $g$ by another Lie algebra $h$. Extensions
arise in several ways. There is a trivial extension obtained by taking a direct sum of
two Lie algebras. Other types are split extension and central extension. Extensions may
arise naturally, for instance, when forming a Lie algebra from projective group representations.
A central extension and an extension by a derivation of a polynomial loop algebra
over a finite-dimensional simple Lie algebra give a Lie algebra which is isomorphic with a
non-twisted affine Kac–Moody algebra \cite[Chapter 19]{bkk}. Using the centrally extended loop
algebra one may construct a current algebra in two spacetime dimensions. The Virasoro
algebra is the universal central extension of the Witt algebra, the Heisenberg algebra is
a central extension of a commutative Lie algebra  \cite[Chapter 18]{bkk}.

The algebraic study of central extensions of Lie and non-Lie algebras has a very big story \cite{omirov,zusmanovich,is11,ha17,hac16,ss78}.
So, Skjelbred and Sund used central extensions of Lie algebras for a classification of nilpotent Lie algebras  \cite{ss78}.
After that, using the method described by Skjelbred and Sund were
described  all central extensions of
 $4$-dimensional Malcev algebras \cite{hac16},
 $3$-dimensional Jordan algebras \cite{ha17},
 $3$-dimensional Zinbiel algebras \cite{ack18},
 $3$-dimensional anticommutative algebras \cite{cfk182},
 $2$-dimensional algebras \cite{cfk18}.
Note that, the method of central extensions is an important tool in the classification of nilpotent algebras
(see for example, \cite{ha16n}).
Using this method, were described
all $4$-dimensional nilpotent associative algebras \cite{degr1},
all $5$-dimensional nilpotent Jordan algebras \cite{ha16},
all $5$-dimensional nilpotent restricted Lie algebras \cite{usefi1},
all $6$-dimensional nilpotent Lie algebras \cite{degr3,degr2},
all $6$-dimensional nilpotent Malcev algebras \cite{hac18}
and some other.

The main results of this paper consist of the classification of central extensions of
null-filiform and filiform complex associative algebras.

This paper is organized as follows. In Section~\ref{S:prel}  we present some basic concepts needed
for this study and give the classification of central extensions of
null-filiform  complex associative algebras. In Section~\ref{S:fil} we  provide the classification of central extensions of
filiform  complex associative algebras and finish with an appendix given the list of obtained main associative algebras.

\section{Preliminaries}\label{S:prel}
\subsection{Null-filiform and filiform associative algebras}

For an algebra $\bf {A}$ of an arbitrary variety, we consider the series
\[
{\bf A}^1={\bf A}, \qquad \ {\bf A}^{i+1}=\sum\limits_{k=1}^{i}{\bf A}^k {\bf A}^{i+1-k}, \qquad i\geq 1.
\]

We say that  an  algebra $\bf A$ is \emph{nilpotent} if ${\bf  A}^{i}=0$ for some $i \in \mathbb{N}$. The smallest integer satisfying ${\bf A}^{i}=0$ is called the  \emph{index of nilpotency} of $\bf A$.

\begin{definition}
An $n$-dimensional algebra $\bf A$ is called null-filiform if $\dim {\bf A}^i=(n+ 1)-i,\ 1\leq i\leq n+1$.
\end{definition}

It is easy to see that an algebra has a maximum nilpotency index if and only if it is null-filiform. For a nilpotent algebra, the condition of null-filiformity is equivalent to the condition that the algebra is one-generated.

All null-filiform associative algebras were described in  \cite{MO,karel}:

\begin{theorem} An arbitrary $n$-dimensional null-filiform associative algebra is isomorphic to the algebra:
\[\mu_0^n : \quad e_i e_j= e_{i+j}, \quad 2\leq i+j\leq n,\]
where $\{ e_1, e_2, \dots, e_n\}$ is a basis of the algebra $\mu_0^n$.
\end{theorem}

\begin{definition}
An $n$-dimensional algebra is called filiform if $\dim (\mathbf{A} ^i)=n-i, \ 2\leq i \leq n$.
\end{definition}

\begin{definition}
An $n$-dimensional associative algebra $\bf{A}$ is called quasi-filiform algebra if ${\bf A}^{n-2}\neq0$ and ${\bf A}^{n-1}=0$.
\end{definition}

\begin{definition}
Given a nilpotent associative algebra $\bf{A}$, put
${\bf A}_i={\bf A}^i/{\bf A}^{i+1}, \ 1 \leq i\leq k-1$, and
$\gr {\bf A} = {\bf A}_1 \oplus {\bf A}_2\oplus\dots \oplus {\bf A}_{k}$.
Then ${\bf A}_i{\bf A}_j\subseteq {\bf A}_{i+j}$ and we
obtain the graded algebra $\gr \bf{A}$. If $\gr \bf{A}$ and $\bf{A}$ are isomorphic,
denoted by $\gr {\bf A} \cong {\bf A}$, we say that the algebra $\bf{A}$ is naturally
graded.
\end{definition}

All filiform and naturally graded quasi-filiform associative algebras were classified in  \cite{KL}.

\begin{theorem}[\cite{KL}]
Every $n$-dimensional ($n>3$) complex filiform associative algebra is isomorphic to one of the next pairwise non-isomorphic algebras with basis $\{ e_1,e_2,\dots,e_n\}$:
\[\begin{array}{lllll}
\mu_{1,1}^n(\mu_0^{n-1}\oplus\mathbb{C}) &: & e_ie_j=e_{i+j},  & & \\
\mu_{1,2}^n &: & e_ie_j=e_{i+j},  & e_ne_n=e_{n-1}, & \\
\mu_{1,3}^n&: & e_ie_j=e_{i+j},  & e_1e_n=e_{n-1}, & \\
\mu_{1,4}^n&: & e_ie_j=e_{i+j},  & e_1e_n=e_{n-1}, & e_ne_n=e_{n-1}
\end{array}\]
where $2\leq i+j\leq n-1$.
\end{theorem}

\begin{theorem}[\cite{KL}]
 Let $\bf{A}$ be $n$-dimensional $(n \geq 6)$ complex naturally graded quasi-filiform  non-split associative algebra. Then it is isomorphic to one of the following pairwise non-isomorphic algebras:
\[\begin{array}{llllll}
\mu_{2,1}^n&: &  e_ie_j=e_{i+j},  &   e_{n-1}e_1=e_n   \\
\mu_{2,2}^n(\alpha)&: & e_ie_j=e_{i+j}, & e_1e_{n-1}=e_n, & e_{n-1}e_1=\alpha e_n \\
\mu_{2,3}^n&: & e_ie_j=e_{i+j}, & e_{n-1}e_{n-1}=e_n & \\
\mu_{2,4}^n&: & e_ie_j=e_{i+j}, & e_1e_{n-1}=e_n, & e_{n-1}e_{n-1}=e_n\\
\end{array}\]
where $\alpha\in\mathbb{C}$ and $2\leq i+j\leq n-2$.
\end{theorem}

\subsection{Basic definitions and methods}
During this paper, we are using the notations and methods well written in \cite{ha17,hac16,cfk18}
and adapted for the associative case with some modifications.
From now, we will give only some important definitions.

Let $({\bf A}, \cdot)$ be an associative algebra over an arbitrary base field $\bf k$ of
characteristic different from $2$ and $\mathbb V$ a vector space over the same base field ${\bf k}$. Then the $\bf k$-linear space $Z^{2}\left(
\bf A,\mathbb V \right) $ is defined as the set of all  bilinear maps $\theta  \colon {\bf A} \times {\bf A} \longrightarrow {\mathbb V}$,
such that \[\theta(xy,z)=\theta(x,yz).\]
 Its elements will be called \emph{cocycles}. For a
linear map $f$ from $\bf A$ to  $\mathbb V$, if we write $\delta f\colon {\bf A} \times
{\bf A} \longrightarrow {\mathbb V}$ by $\delta f  (x,y ) =f(xy )$, then $\delta f\in Z^{2}\left( {\bf A},{\mathbb V} \right) $. We define $B^{2}\left(
{\bf A},{\mathbb V}\right) =\left\{ \theta =\delta f\ : f\in \Hom\left( {\bf A},{\mathbb V}\right) \right\} $.
One can easily check that $B^{2}(\bf A,\mathbb V)$ is a linear subspace of $Z^{2}\left( {\bf A},{\mathbb V}\right) $ which elements are called
\emph{coboundaries}. We define the \emph{second cohomology space} $H^{2}\left( {\bf A},{\mathbb V}\right) $ as the quotient space $Z^{2}
\left( {\bf A},{\mathbb V}\right) \big/B^{2}\left( {\bf A},{\mathbb V}\right) $.

\

Let $\Aut(\mathbf{A}) $ be the automorphism group of the associative  algebra $\mathbf{A} $ and let $\phi \in \Aut(\mathbf{A})$. For $\theta \in
Z^{2}\left( {\bf A},{\mathbb V}\right) $ define $\phi \theta (x,y)
=\theta \left( \phi \left( x\right) ,\phi \left( y\right) \right) $. Then $\phi \theta \in Z^{2}\left( {\bf A},{\mathbb V}\right) $. So, $\Aut(\mathbf{A})$
acts on $Z^{2}\left( {\bf A},{\mathbb V}\right) $. It is easy to verify that
 $B^{2}\left( {\bf A},{\mathbb V}\right) $ is invariant under the action of $\Aut(\mathbf{A})$ and so we have that $\Aut(\mathbf{A})$ acts on $H^{2}\left( {\bf A},{\mathbb V}\right)$.

\

Let $\bf A$ be an associative  algebra of dimension $m<n$ over an arbitrary base field $\bf k$ of characteristic different from $2$, and ${\mathbb V}$ be a $\bf k$-vector
space of dimension $n-m$. For any $\theta \in Z^{2}\left(
{\bf A},{\mathbb V}\right) $ define on the linear space ${\bf A}_{\theta } \coloneqq {\bf A}\oplus {\mathbb V}$ the
bilinear product `` $\left[ -,-\right] _{{\bf A}_{\theta }}$'' by $\left[ x+x^{\prime },y+y^{\prime }\right] _{{\bf A}_{\theta }}=
 xy +\theta(x,y) $ for all $x,y\in {\bf A},x^{\prime },y^{\prime }\in {\mathbb V}$.
The algebra ${\bf A}_{\theta }$ is an associative algebra which is called an $(n-m)$-\emph{dimensional central extension} of ${\bf A}$ by ${\mathbb V}$. Indeed, we have, in a straightforward way, that ${\bf A_{\theta}}$ is an associative algebra if and only if $\theta \in Z^2({\bf A}, {\mathbb V})$.

We also call to the
set $\Ann(\theta)=\left\{ x\in {\bf A}:\theta \left( x, {\bf A} \right)+ \theta \left({\bf A} ,x\right) =0\right\} $
the \emph{annihilator} of $\theta $. We recall that the \emph{annihilator} of an  algebra ${\bf A}$ is defined as
the ideal $\Ann(  \mathbf{A} ) =\left\{ x\in {\bf A}:  x{\bf A}+ {\bf A}x =0\right\}$ and observe
 that
$\Ann\left( {\bf A}_{\theta }\right) = \Ann(\theta) \cap \Ann(\mathbf{A})
 \oplus {\mathbb V}$.

\

We have the next  key result:

\begin{lemma}
Let ${\bf A}$ be an $n$-dimensional associative algebra such that $\dim(\Ann({\bf A}))=m\neq0$. Then there exists, up to isomorphism, a unique $(n-m)$-dimensional associative algebra $\mathbf{A}'$ and a bilinear map $\theta \in Z^2({\bf A}, {\mathbb V})$ with $\Ann({\bf A})\cap \Ann(\theta)=0$, where $\mathbb V$ is a vector space of dimension m, such that $\mathbf{A} \cong {\mathbf{A}'}_{\theta}$ and
 ${\bf A}/\Ann({\bf A})\cong \mathbf{A}'$.
\end{lemma}

\begin{proof}
Let $\mathbf{A}'$ be a linear complement of $\Ann({\bf A})$ in ${\bf A}$. Define a linear map $P \colon {\bf A} \longrightarrow \mathbf{A}'$ by $P(x+v)=x$ for $x\in \mathbf{A}'$ and $v\in \Ann({\bf A})$ and define a multiplication on $\mathbf{A}'$ by $[x, y]_{\mathbf{A}'}=P(x y)$ for $x, y \in \mathbf{A}'$.
For $x, y \in {\bf A}$, we have
\[P(xy)=P((x-P(x)+P(x))(y- P(y)-P(y)))=P(P(x) P(y))=[P(x), P(y)]_{\mathbf{A}'}. \]

Since $P$ is a homomorphism we have that $P({\bf A})=\mathbf{A}'$ is an associative algebra and
 ${\bf A}/\Ann({\bf A})\cong \mathbf{A}'$, which give us the uniqueness. Now, define the map $\theta \colon \mathbf{A}' \times \mathbf{A}' \longrightarrow \Ann({\bf A})$ by $\theta(x,y)=xy- [x,y]_{\mathbf{A}'}$.
  Thus, $\mathbf{A}'_{\theta}$ is ${\bf A}$ and therefore $\theta \in Z^2({\bf A}, {\mathbb V})$ and $\Ann({\bf A})\cap \Ann(\theta)=0$.
\end{proof}

\

However, in order to solve the isomorphism problem we need to study the
action of $\Aut(\mathbf{A})$ on $H^{2}\left( {\bf A},{\mathbb V}
\right) $. To do that, let us fix $e_{1},\ldots ,e_{s}$ a basis of ${\mathbb V}$, and $
\theta \in Z^{2}\left( {\bf A},{\mathbb V}\right) $. Then $\theta $ can be uniquely
written as $\theta \left( x,y\right) =
\displaystyle \sum_{i=1}^{s} \theta _{i}\left( x,y\right) e_{i}$, where $\theta _{i}\in
Z^{2}\left( {\bf A},\bf k\right) $. Moreover, $\Ann(\theta)=\Ann(\theta _{1})\cap \Ann(\theta _{2})\cdots \cap \Ann(\theta _{s})$. Further, $\theta \in
B^{2}\left( {\bf A},{\mathbb V}\right) $\ if and only if all $\theta _{i}\in B^{2}\left( {\bf A},
\bf k\right) $.

\;

Given an associative algebra ${\bf A}$, if ${\bf A}=I\oplus \bf kx$
is a direct sum of two ideals, then $\bf kx$ is called an \emph{annihilator component} of ${\bf A}$.

\begin{definition}
A central extension of an algebra $\bf A$ without annihilator component is called a non-split central extension.
\end{definition}

It is not difficult to prove (see \cite[Lemma 13]{hac16}), that given an associative algebra ${\bf A}_{\theta}$, if we write as
above $\theta \left( x,y\right) = \displaystyle \sum_{i=1}^{s} \theta_{i}\left( x,y\right) e_{i}\in Z^{2}\left( {\bf A},{\mathbb V}\right) $ and we have
$\Ann(\theta)\cap \Ann\left( {\bf A}\right) =0$, then ${\bf A}_{\theta }$ has an
annihilator component if and only if $\left[ \theta _{1}\right] ,\left[
\theta _{2}\right] ,\ldots ,\left[ \theta _{s}\right] $ are linearly
dependent in $H^{2}\left( {\bf A},\bf k\right) $.

\;

Let ${\mathbb V}$ be a finite-dimensional vector space over $\bf k$. The \emph{Grassmannian} $G_{k}\left( {\mathbb V}\right) $ is the set of all $k$-dimensional
linear subspaces of $ {\mathbb V}$. Let $G_{s}\left( H^{2}\left( {\bf A},\bf k\right) \right) $ be the Grassmannian of subspaces of dimension $s$ in
$H^{2}\left( {\bf A},\bf k\right) $. There is a natural action of $\Aut(\mathbf{A})$ on $G_{s}\left( H^{2}\left( {\bf A},\bf k\right) \right) $.
Let $\phi \in \Aut(\mathbf{A})$. For $W=\left\langle
\left[ \theta _{1}\right] ,\left[ \theta _{2}\right] ,\dots,\left[ \theta _{s}
\right] \right\rangle \in G_{s}\left( H^{2}\left( {\bf A},\bf k
\right) \right) $ define $\phi W=\left\langle \left[ \phi \theta _{1}\right]
,\left[ \phi \theta _{2}\right] ,\dots,\left[ \phi \theta _{s}\right]
\right\rangle $. Then $\phi W\in G_{s}\left( H^{2}\left( {\bf A},\bf k \right) \right) $. We denote the orbit of $W\in G_{s}\left(
H^{2}\left( {\bf A},\bf k\right) \right) $ under the action of $\Aut(\mathbf{A})$ by $\Orb(W)$. Since given
\[
W_{1}=\left\langle \left[ \theta _{1}\right] ,\left[ \theta _{2}\right] ,\dots,
\left[ \theta _{s}\right] \right\rangle ,W_{2}=\left\langle \left[ \vartheta
_{1}\right] ,\left[ \vartheta _{2}\right] ,\dots,\left[ \vartheta _{s}\right]
\right\rangle \in G_{s}\left( H^{2}\left( {\bf A},\bf k\right)
\right),
\]
we easily have that in case $W_{1}=W_{2}$, then $ \bigcap\limits_{i=1}^{s} \Ann(\theta _{i})\cap \Ann\left( {\bf A}\right) = \bigcap\limits_{i=1}^{s}
 \Ann(\vartheta _{i})\cap \Ann( {\bf A}) $, and so we can introduce
the set
\[
T_{s}(\mathbf{A}) =\left\{ W=\left\langle \left[ \theta _{1}\right] ,
\left[ \theta _{2}\right] ,\dots,\left[ \theta _{s}\right] \right\rangle \in
G_{s}\left( H^{2}\left( {\bf A},\bf k\right) \right) : \bigcap\limits_{i=1}^{s}\Ann(\theta _{i})\cap \Ann(\mathbf{A}) =0\right\},
\]
which is stable under the action of $\Aut(\mathbf{A})$.

\

Now, let ${\mathbb V}$ be an $s$-dimensional linear space and let us denote by
$E\left( {\bf A},{\mathbb V}\right) $ the set of all \emph{non-split $s$-dimensional central extensions} of ${\bf A}$ by
${\mathbb V}$. We can write
\[
E\left( {\bf A},{\mathbb V}\right) =\left\{ {\bf A}_{\theta }:\theta \left( x,y\right) = \sum_{i=1}^{s}\theta _{i}\left( x,y\right) e_{i} \ \ \text{and} \ \ \left\langle \left[ \theta _{1}\right] ,\left[ \theta _{2}\right] ,\dots,
\left[ \theta _{s}\right] \right\rangle \in T_{s}(\mathbf{A}) \right\} .
\]
We also have the next result, which can be proved as \cite[Lemma 17]{hac16}.

\begin{lemma}
 Let ${\bf A}_{\theta },{\bf A}_{\vartheta }\in E\left( {\bf A},{\mathbb V}\right) $. Suppose that $\theta \left( x,y\right) =  \displaystyle \sum_{i=1}^{s}
\theta _{i}\left( x,y\right) e_{i}$ and $\vartheta \left( x,y\right) =
\displaystyle \sum_{i=1}^{s} \vartheta _{i}\left( x,y\right) e_{i}$.
Then the associative algebras ${\bf A}_{\theta }$ and ${\bf A}_{\vartheta } $ are isomorphic
if and only if $\Orb\left\langle \left[ \theta _{1}\right] ,
\left[ \theta _{2}\right] ,\dots,\left[ \theta _{s}\right] \right\rangle =
\Orb\left\langle \left[ \vartheta _{1}\right] ,\left[ \vartheta
_{2}\right] ,\dots,\left[ \vartheta _{s}\right] \right\rangle $.
\end{lemma}

From here, there exists a one-to-one correspondence between the set of $\Aut(\mathbf{A})$-orbits on $T_{s}\left( {\bf A}\right) $ and the set of
isomorphism classes of $E\left( {\bf A},{\mathbb V}\right) $. Consequently we have a
procedure that allows us, given the associative algebra $\mathbf{A}'$ of
dimension $n-s$, to construct all non-split central extensions of $\mathbf{A}'$. This procedure would be:

\; \;

{\centerline {\textsl{Procedure}}}

\begin{enumerate}
\item For a given associative algebra $\mathbf{A}'$ of dimension $n-s $, determine $H^{2}( \mathbf{A}',\mathbf {k}) $, $\Ann(\mathbf{A}')$ and $\Aut(\mathbf{A}')$.

\item Determine the set of $\Aut(\mathbf{A}')$-orbits on $T_{s}(\mathbf{A}') $.

\item For each orbit, construct the associative algebra corresponding to a
representative of it.
\end{enumerate}

\

Finally, let us introduce some of notation. Let ${\bf A}$ be an associative algebra with
a basis $e_{1},e_{2},\dots,e_{n}$. Then by $\Delta _{i,j}$\ we will denote the
associative bilinear form
$\Delta _{i,j} \colon {\bf A}\times {\bf A}\longrightarrow \bf k$
with $\Delta _{i,j}\left( e_{l},e_{m}\right) = \delta_{il}\delta_{jm}$.
Then the set $\left\{ \Delta_{i,j}:1\leq i, j\leq n\right\} $ is a basis for the linear space of
the bilinear forms on ${\bf A}$. Then every $\theta \in
Z^{2}\left( {\bf A},\bf k\right) $ can be uniquely written as $
\theta = \displaystyle \sum_{1\leq i,j\leq n} c_{ij}\Delta _{{i},{j}}$, where $
c_{ij}\in \bf k$.

\subsection{Central extension of null-filiform associative algebras}

Using the presented procedure we can easy find all central extensions of the  null-filiform associative algebras.
We need the following  proposition.
\begin{proposition} Let $\mu_0^n$ be the null-filiform associative algebra. Then
\[Z^2(\mu_0^n,\mathbb{C})=\langle\sum\limits_{j=1}^{i-1}\Delta_{j,i-j}, \ 2\leq i\leq n+1\rangle,\quad B^2(\mu_0^n,\mathbb{C})=\langle\sum\limits_{j=1}^{i-1}\Delta_{j,i-j}, \ 2\leq i\leq n\rangle,\]
\[H^2(\mu_0^n,\mathbb{C})=Z^2(\mu_0^n,\mathbb{C})/B^2(\mu_0^n,\mathbb{C})=\langle [ \sum\limits_{j=1}^{n}\Delta_{j,n+1-j}] \rangle,\]
\[\phi_0^n=
\begin{pmatrix}
 a_{1,1} &  &   & &  \\
  & a_{1,1}^2 & &0  &  \\
  &  & a_{1,1}^3 &  &  \\
  &  \ast & & \ddots &  \\
  &  &  &  & a_{1,1}^n
  \end{pmatrix},
   \quad \phi_0^n\in \Aut(\mu_0^n).\]
\end{proposition}
\begin{proof}
The proof follows directly from the definition of a cocycle and an automorphism.
\end{proof}

As $\dim(H^2(\mu_0^n,\mathbb{C}))=1$, it is easy to see the following result.
\begin{theorem}A central extension of an $n$-dimensional complex null-filiform associative algebra is isomorphic to
$\mu_0^{n+1}$.
\end{theorem}

\section{Central extension of filiform associative algebras}\label{S:fil}

\begin{proposition} Let $\mu_{1,1}^n,\mu_{1,2}^n,\mu_{1,3}^n$ and $\mu_{1,4}^n$ be $n$-dimensional complex filiform associative algebras. Then:

\begin{itemize}
  \item A basis of $Z^2(\mu_{1,s}^n,\mathbb{C})$ is formed by the following cocycles
\[\begin{array}{l}
  Z^2(\mu_{1,1}^n,\mathbb{C})=\langle\sum\limits_{j=1}^{i-1}\Delta_{j,i-j}, \Delta_{1,n},\Delta_{n,1},\Delta_{n,n}, \ 2\leq i\leq n\rangle, \\
  Z^2(\mu_{1,k}^n,\mathbb{C})=\langle\sum\limits_{j=1}^{i-1}\Delta_{j,i-j}, \Delta_{1,n},\Delta_{n,1},\Delta_{n,n}, \ 2\leq i\leq n-1\rangle, \quad 2\leq k\leq4;
\end{array}\]
  \item A basis of $B^2(\mu_{1,s}^n,\mathbb{C})$  is formed by the following coboundaries
\[\begin{array}{l}
  B^2(\mu_{1,1}^n,\mathbb{C})=\langle\sum\limits_{j=1}^{i-1}\Delta_{j,i-j}, \ 2\leq i\leq n-1\rangle, \\
  B^2(\mu_{1,2}^n,\mathbb{C})=\langle\sum\limits_{j=1}^{i-1}\Delta_{j,i-j}, \sum\limits_{j=1}^{n-2}\Delta_{j,n-j-1}+\Delta_{n,n}, \ 2\leq i\leq n-2\rangle,\\
  B^2(\mu_{1,3}^n,\mathbb{C})=\langle\sum\limits_{j=1}^{i-1}\Delta_{j,i-j}, \sum\limits_{j=1}^{n-2}\Delta_{j,n-j-1}+\Delta_{1,n}, \ 2\leq i\leq n-2\rangle,\\
  B^2(\mu_{1,4}^n,\mathbb{C})=\langle\sum\limits_{j=1}^{i-1}\Delta_{j,i-j}, \sum\limits_{j=1}^{n-2}\Delta_{j,n-j-1}+\Delta_{1,n}+\Delta_{n,n}, \ 2\leq i\leq n-2\rangle;
\end{array}\]
  \item A basis of $H^2(\mu_{1,s}^n,\mathbb{C})$ is formed by the following cocycles
\[\begin{array}{l}
  H^2(\mu_{1,1}^n,\mathbb{C})=\langle [\sum\limits_{j=1}^{n-1}\Delta_{j,n-j}], [\Delta_{1,n}], [\Delta_{n,1}], [\Delta_{n,n}]\rangle, \\
  H^2(\mu_{1,k}^n,\mathbb{C})=\langle [\Delta_{1,n}],[\Delta_{n,1}],[\Delta_{n,n}] \rangle, \quad 2\leq k\leq4.
\end{array}\]
\end{itemize}
\end{proposition}
\begin{proof}
The proof follows directly from the definition of a cocycle.
\end{proof}

\begin{proposition}
 Let $\phi_{1,s}^n\in \Aut(\mu_{1,s}^n)$. Then

 \begin{align*}
\phi_{1,1}^n= \begin{pmatrix}
       &  &  &  &  & 0 \\
       &  &  &  &  & 0 \\
       &  & \phi_0^{n-1} &  &  & \vdots \\
       &  &  &  &  & 0 \\
       &  &  &  &  & a_{n-1,n} \\
      a_{n,1} & 0 & \dots & 0 & 0 & a_{n,n} \\
    \end{pmatrix}, & \ \phi_{1,2}^n=
    \begin{pmatrix}
       &  &  &  &  & 0 \\
       &  &  &  &  & \vdots \\
       &  & \phi_0^{n-1} &  &  & 0 \\
       &  &  &  &  & -a_{n,1}a_{1,1}^{(n-3)/2} \\
       &  &  &  &  & a_{n-1,n} \\
      a_{n,1} & 0 & \dots & 0 & 0 & a_{1,1}^{(n-1)/2} \\
    \end{pmatrix} \\[2mm]
    \phi_{1,3}^n=
    \begin{pmatrix}
       &  &  &  &  & 0 \\
       &  &  &  &  & 0 \\
       &  & \phi_0^{n-1} &  &  & \vdots \\
       &  &  &  &  & 0 \\
       &  &  &  &  & a_{n-1,n} \\
      a_{n,1} & 0 & \dots & 0 & 0 & a_{1,1}^{n-2} \\
    \end{pmatrix}, &  \ \phi_{1,4}^n=
    \begin{pmatrix}
       &  &  &  &  & 0 \\
       &  &  &  &  & \vdots \\
       &  & \phi_0^{n-1}(1) &  &  & 0 \\
       &  &  &  &  & -a_{n,1} \\
       &  &  &  &  & a_{n-1,n} \\
      a_{n,1} & 0 & \dots & 0 & 0 & 1 \\
    \end{pmatrix}
\end{align*}
where $\phi_0^{n-1}(1)$ is $\phi_0^{n-1}$ with $a_{1,1}=1$.
\end{proposition}

\subsection{Central extensions of $\mu_{1,1}^n$}

Let us denote
\[\nabla_1=\sum\limits_{j=1}^{n-1}[\Delta_{j,n-j}], \   \nabla_2= [\Delta_{1,n}], \ \nabla_3= [\Delta_{n,1}], \  \nabla_4= [\Delta_{n,n}]\in H^2(\mu_{1,1}^n,\mathbb{C})\]
and $x=a_{1,1},y=a_{n,n}, z=a_{n-1,n},w=a_{n,1}$.
Since \[\left(
  \begin{array}{ccccc}
    \ast & \dots & \ast & \alpha_1^\prime & \alpha_2^\prime \\
    \ast & \dots & \alpha_1^\prime & 0 & 0 \\
    \vdots & \ldots & \vdots & \vdots & \vdots \\
    \alpha_1^\prime & \dots & 0 & 0 & 0 \\
    \alpha_3^\prime & \dots & 0 & 0 & \alpha_4^\prime \\
  \end{array}
\right)=(\phi_{1,1}^n)^T\left(
  \begin{array}{ccccc}
    0 & \dots & 0 & \alpha_1 & \alpha_2 \\
    0 & \dots & \alpha_1 & 0 & 0 \\
    \vdots & \ldots & \vdots & \vdots & \vdots \\
    \alpha_1 & \dots & 0 & 0 & 0 \\
    \alpha_3 & \dots & 0 & 0 & \alpha_4 \\
  \end{array}
\right)\phi_{1,1}^n,\]
for any
$\theta=\alpha_1 \nabla_1 +\alpha_2 \nabla_2 +\alpha_3 \nabla_3+\alpha_4 \nabla_4$, we have
the action of the automorphism group on the subspace $\langle \theta \rangle$ as
\[\langle
\alpha_1 x^n \nabla_1 + (\alpha_2 xy+\alpha_1 xz+\alpha_4 wy) \nabla_2 +
                        (\alpha_3 xy+\alpha_1 xz+\alpha_4 wy) \nabla_3+  \alpha_4 y^2 \nabla_4\rangle.\]

\subsubsection{$1$-dimensional central extensions of $\mu_{1,1}^n$}
Let us consider the following cases:

\begin{enumerate}
    \item[(a)] $\alpha_1 =0$.
    \begin{enumerate}
        \item[(1)] $\alpha_4=0, \alpha_2=0, \alpha_3 \neq 0$.
        Choosing $x= 1/ \alpha_3, y=1$, we have a representative $\nabla_3$ and the orbit is $\langle \nabla_3 \rangle$.

        \item[(2)] $\alpha_4=0, \alpha_2\neq0$.
        Choosing $x= 1/ \alpha_2, y=1$ and $\alpha= \alpha_3 / \alpha_2$, we have a representative $\nabla_2+\alpha \nabla_3$ and the orbit is $\langle \nabla_2+\alpha \nabla_3 \rangle$.

        \item[(3)] $\alpha_4\neq 0, \alpha_2=\alpha_3$.
        Choosing $y=1/\sqrt{\alpha_4}, w= - \frac{\alpha_2}{\alpha_4}, x=1$, we have a representative $ \nabla_4$ and the orbit is $\langle \nabla_4 \rangle$.

        \item[(4)] $\alpha_4\neq 0, \alpha_2 \neq \alpha_3$.
        Choosing
        $x=\frac{\sqrt{\alpha_4}}{\alpha_2-\alpha_3}, y=\frac{1}{\sqrt{\alpha_4}},w=  \frac{\alpha_3}{\sqrt{\alpha_4}(\alpha_3-\alpha_2)}$,
        we have a representative $\nabla_3+\nabla_4$ and the orbit is $\langle \nabla_3+ \nabla_4 \rangle$.
    \end{enumerate}

    \item[(b)] $\alpha_1 \neq 0$.

\begin{enumerate}
        \item[(1)] $\alpha_4=0, \alpha_2=\alpha_3$.
        Choosing $x =  \frac{1}{\sqrt[n]{\alpha_1}}, y=1, z= - \frac{\alpha_2}{\alpha_1}$, we have a representative $\nabla_1$ and the orbit is
        $\langle \nabla_1 \rangle$.

        \item[(2)] $\alpha_4=0, \alpha_2\neq \alpha_3$.
        Choosing $x= \frac{1}{\sqrt[n]{\alpha_1}}, y=\frac{\sqrt[n]{\alpha_1}}{\alpha_3-\alpha_2}, z=-\frac{\alpha_2}{\alpha_1}$,
        we have a representative $\nabla_1+ \nabla_3$ and the orbit is $\langle \nabla_1+ \nabla_3 \rangle$.

        \item[(3)] $\alpha_4\neq 0, \alpha_2=\alpha_3$.
        Choosing $x =  \frac{1}{\sqrt[n]{\alpha_1}}, y=\frac{1}{\sqrt{\alpha_4}},
        z= - \frac{\alpha_3}{\alpha_1 \sqrt{\alpha_4}}, w=0$, we have a representative $\nabla_1+\nabla_4$ and the orbit is $\langle \nabla_1 +\nabla_4\rangle$.

        \item[(4)] $\alpha_4\neq 0, \alpha_2 \neq \alpha_3$.
        Choosing
        $ x =  \sqrt[\leftroot{-3}\uproot{3}n-2]{\frac{(\alpha_3-\alpha_2)^2}{\alpha_1\alpha_4}}$,
          $y= \frac{\alpha_3-\alpha_2}{\alpha_4}\sqrt[\leftroot{-3}\uproot{3}n-2]{\frac{(\alpha_3-\alpha_2)^2}{\alpha_1\alpha_4}}$,
          $z= -\frac{\alpha_2(\alpha_3-\alpha_2)}{\alpha_1\alpha_4}\sqrt[\leftroot{-3}\uproot{3}n-2]{\frac{(\alpha_3-\alpha_2)^2}{\alpha_1\alpha_4}}$,
         $ w=0 $,
        we have a representative
         \[ \alpha_4 \bigg(\frac{\alpha_3-\alpha_2}{\alpha_4}\sqrt[\leftroot{-3}\uproot{3}n-2]{\frac{(\alpha_3-\alpha_2)^2}{\alpha_1\alpha_4}}\bigg)^2(\nabla_1+  \nabla_3+\nabla_4) \]
         and the orbit is $\langle \nabla_1+  \nabla_3+\nabla_4 \rangle$.
    \end{enumerate}
\end{enumerate}

It is easy to verify that all previous orbits  are different, and so we obtain
\begin{align*}
T_1(\mu_{1,1}^n)&=
\Big \langle \nabla_1 \Big \rangle \cup
\Big \langle \nabla_1+ \nabla_3 \Big \rangle \cup
\Big \langle \nabla_1+ \nabla_3+\nabla_4 \Big \rangle \cup
\Big \langle \nabla_1+ \nabla_4 \Big \rangle  \\
&  \qquad  \! \! \qquad \cup \Big \langle \nabla_2+ \alpha \nabla_3 \Big \rangle \cup
\Big \langle \nabla_3  \Big \rangle \cup
\Big \langle \nabla_3+ \nabla_4 \Big \rangle \cup
\Big \langle \nabla_4  \Big \rangle.
\end{align*}

\subsubsection{$2$-dimensional central extensions of $\mu_{1,1}^n$}
We may assume that a $2$-dimensional subspace is generated by
\begin{align*}
\theta_1 & = \alpha_1 \nabla_1+ \alpha_2 \nabla_2+\alpha_3 \nabla_3 +\alpha_4 \nabla_4,\\
\theta_2 & = \beta_1 \nabla_1+ \beta_2 \nabla_2 +\beta_3 \nabla_3.
\end{align*}

Then we have the two following cases:

(a) $\alpha_4 \neq 0$.
    \begin{enumerate}
        \item[(1)] $\alpha_1=0, \beta_1\neq0, \alpha_2=\alpha_3, \beta_2=\beta_3$.
   Choosing $x =  \frac{1}{\sqrt[n]{\beta_1}}$, $y=\frac{1}{\sqrt{\alpha_4}}$,
   $z= - \frac{\beta_2}{\beta_1 \sqrt{\alpha_4}}$, $w= -\frac{\alpha_2}{\alpha_4\sqrt[n]{\beta_1}}$, we have a representative
   $\{ \nabla_4, \nabla_1 \}$ and the orbit is $\langle \nabla_1, \nabla_4 \rangle$.

        \item[(2)] $\alpha_1=0, \beta_1=0, \alpha_2=\alpha_3, \beta_2=\beta_3$.
We have a representative
   $\{ \alpha_2(\nabla_2+\nabla_3) + \alpha_4\nabla_4, \beta_2(\nabla_2+\nabla_3) \}$ and the orbit is $\langle \nabla_2+\nabla_3, \nabla_4 \rangle$.

        \item[(3)] $\alpha_1=0, \beta_1\neq0, \alpha_2 \neq \alpha_3, \beta_2=\beta_3$.
   Choosing $y=\frac{\alpha_3-\alpha_2}{\alpha_4}x, z=-\frac{\beta_2(\alpha_3-\alpha_2)}{\beta_1\alpha_4}x$, $w=-\frac{\beta_2}{\beta_4}x$, we have a representative
   $\{ \alpha_4y^2(\nabla_3+\nabla_4), \nabla_1 \}$ and the orbit is $\langle \nabla_1, \nabla_3+\nabla_4 \rangle$.

        \item[(4)] $\alpha_1=0, \beta_1=0, \alpha_2 \neq \alpha_3, \beta_2=\beta_3$.
   Choosing $y =  \frac{(\alpha_3-\alpha_2)}{\alpha_4}x,
             w=   -\frac{\alpha_2}{\alpha_4}x$, we have a representative
   $\{ \alpha_4 y^2(\nabla_3+\nabla_4), \beta_2xy(\nabla_2+\nabla_3) \}$ and the orbit is
   $\langle \nabla_2+\nabla_3, \nabla_3+\nabla_4 \rangle$.

        \item[(5)] $\alpha_1=0, \beta_1=0, \alpha_2 = \alpha_3, \beta_2=0,  \beta_3\neq0$.
   Choosing $x=1, y=\frac{1}{\sqrt{\alpha_4}}, w = -\frac{\alpha_2}{\alpha_4}$, we have a representative
   $\{ \nabla_4, \nabla_3 \}$ and the orbit is
   $\langle \nabla_3, \nabla_4 \rangle$.

        \item[(6)] $\alpha_1=0, \beta_1=0, \alpha_2 = \alpha_3, \beta_2\neq 0, \beta_2 \neq \beta_3$.
   Choosing $w =  -\frac{\alpha_2}{\alpha_4}x$ and $\alpha=\frac{\beta_3}{\beta_2}$, we have a representative
   $\{ \alpha_4y^2\nabla_4, \beta_2xy(\nabla_2 +\alpha  \nabla_3) \}_{\alpha\neq 1}$ and the orbit is
   $\langle \nabla_2 +\alpha  \nabla_3, \nabla_4 \rangle_{\alpha\neq 1}$.

        \item[(7)] $\alpha_1=0, \beta_1\neq 0, \alpha_2 = \alpha_3, \beta_2 \neq \beta_3$.
   Choosing $y=\frac{\beta_1}{\beta_2-\beta_3}x^{n-1},
   z=-\frac{\beta_3}{\beta_1}x,
   w=-\frac{\alpha_2}{\alpha_4}x$,  we have a representative
   $\{ \alpha_4y^2\nabla_4, \beta_1x^n(\nabla_1 +\nabla_2) \}$ and the orbit is
   $\langle \nabla_1 +  \nabla_2, \nabla_4 \rangle$.

        \item[(8)] $\alpha_1=0, \beta_1= 0, \alpha_2 \neq \alpha_3, \beta_2 \neq \beta_3$.
        We can choose a representative as in the case (a.5) or (a.6).

        \item[(9)] $\alpha_1=0, \beta_1\neq 0, \alpha_2 \neq \alpha_3, \beta_2 \neq \beta_3$.
   Choosing $x =  \sqrt[\leftroot{-3}\uproot{3}n-2]{\frac{(\beta_2-\beta_3)(\alpha_3-\alpha_2)}{\alpha_4\beta_1}},
             y=   \frac{\alpha_3-\alpha_2}{\alpha_4}x,
   z= -\frac{\beta_3(\alpha_3-\alpha_2)}{\beta_1\alpha_4}x,
   w= -\frac{\alpha_4}{\alpha_2}x$,  we have a representative
   $\{ \alpha_4y^2(\nabla_3+\nabla_4), \beta_1x^n(\nabla_1 +\nabla_2) \}$ and the orbit is
   $\langle \nabla_1 +  \nabla_2, \nabla_3+\nabla_4 \rangle$.

        \item[(10)] $\alpha_1\neq 0, \beta_1= 0, \alpha_2 = \alpha_3, \beta_2= 0,  \beta_3\neq 0$.
   Choosing $x =  \frac{1}{\sqrt[n]{\alpha_1}} ,
             y=   \frac{1}{\sqrt{\alpha_4}},
             z= -\frac{\alpha_2}{\alpha_4 \sqrt{\alpha_4}}, w= 0$,  we have a representative
   $\{ \nabla_1+ \nabla_4, \beta_3xy \nabla_3 \}$ and the orbit is
   $\langle \nabla_1 +  \nabla_4,  \nabla_3 \rangle$.

        \item[(11)] $\alpha_1\neq 0, \beta_1= 0, \alpha_2 = \alpha_3, \beta_2\neq 0$.
   Choosing $x =  \frac{1}{\sqrt[n]{\alpha_1}}$ ,
             $y=   \frac{1}{\sqrt{\alpha_4}}$,
             $z= -\frac{\alpha_2}{\alpha_4 \sqrt{\alpha_4}}$, $w= 0$  and $\alpha=\frac{\beta_3}{\beta_2}$,
             we have a representative
   $\{ \nabla_1+ \nabla_4, \beta_2xy (\nabla_2 +\alpha \nabla_3) \}$ and the orbit is
   $\langle \nabla_1 +  \nabla_4,  \nabla_2 +\alpha\nabla_3 \rangle$.

        \item[(12)] $\alpha_1\neq 0, \beta_1= 0, \alpha_2 \neq \alpha_3, \beta_2=\beta_3$.

   Choosing $x =  \sqrt[\leftroot{-3}\uproot{3}n-2]{\frac{(\alpha_2-\alpha_3)^2}{(\alpha_n)^2}}$,
             $y=   \frac{\alpha_2-\alpha_3}{\alpha_4}\sqrt[\leftroot{-3}\uproot{3}n-2]{\frac{(\alpha_2-\alpha_3)^2}{(\alpha_4)^2}}$,
            $ z=  0$,
            $ w=   -\frac{\alpha_3}{\alpha_4}x$,
             we have a representative
   $\{ \alpha_1x^n(\nabla_1+ \nabla_2+\nabla_4), \beta_2xy (\nabla_2 +\nabla_3) \}$ and the orbit is
   $\langle \nabla_1+ \nabla_2+\nabla_4, \nabla_2 + \nabla_3 \rangle$.

     \item[(13)] $\alpha_1\neq 0, \beta_1= 0, \alpha_2 \neq \alpha_3, \beta_2 \neq \beta_3$.
           We can choose a representative as in the case (a.1) or (a.11).
    \end{enumerate}

   (b) $\alpha_4 = 0, \beta_3=0$.
\begin{enumerate}

         \item[(1)] $\alpha_1\neq 0, \beta_1= 0, \alpha_2 = \alpha_3$.
   It is easy to see that the orbit is $\langle \nabla_1, \nabla_3 \rangle$
   or          $\langle \nabla_1, \nabla_2+ \alpha \nabla_3 \rangle$.

         \item[(2)] $\alpha_1\neq 0, \beta_1= 0, \alpha_2 \neq \alpha_3, \beta_2 \neq \beta_3$.
   It is easy to see that we have the case (b.1).

         \item[(3)] $\alpha_1\neq 0, \beta_1= 0, \alpha_2 \neq \alpha_3, \beta_2=\beta_3$.
   It is easy to see that the orbit is $\langle \nabla_1+\nabla_2 , \nabla_2+\nabla_3 \rangle$.

        \item[(4)] $\alpha_1= 0, \beta_1= 0$.
   It is easy to see that the orbit is $\langle \nabla_2, \nabla_3 \rangle$.

        \item[(5)] $\alpha_3= 0$.
   It is easy to see that we have the case (b.1).

\end{enumerate}

It is easy to verify that all previous orbits   are different, and so we obtain
\begin{align*}
T_2(\mu_{1,1}^n) & =
\Big \langle \nabla_1, \nabla_2+\alpha \nabla_3 \Big \rangle  \cup
\Big \langle \nabla_1, \nabla_3 \Big \rangle  \cup
\Big \langle \nabla_1, \nabla_3+\nabla_4 \Big \rangle  \cup
\Big \langle \nabla_1, \nabla_4 \Big \rangle  \\
& \quad \cup
\Big \langle \nabla_1+ \nabla_2, \nabla_2+\nabla_3 \Big \rangle  \cup
\Big \langle \nabla_1+ \nabla_2, \nabla_3+\nabla_4 \Big \rangle  \cup
\Big \langle \nabla_1+ \nabla_2, \nabla_4 \Big \rangle    \\
 &  \quad \cup  \Big \langle \nabla_1+ \nabla_2+\nabla_4, \nabla_2 + \nabla_3 \Big \rangle  \cup
\Big \langle \nabla_1+ \nabla_4, \nabla_2 +\alpha\nabla_3 \Big \rangle  \cup
\Big \langle \nabla_1+ \nabla_4, \nabla_3 \Big \rangle     \\
& \quad \cup \Big \langle \nabla_2, \nabla_3 \Big \rangle \cup
\Big \langle \nabla_2+ \nabla_3, \nabla_3+\nabla_4 \Big \rangle  \cup
\Big \langle \nabla_2 +\alpha  \nabla_3, \nabla_4 \Big \rangle  \cup
\Big \langle \nabla_3, \nabla_4 \Big \rangle.
\end{align*}

\subsubsection{$3$-dimensional central extensions of $\mu_{1,1}^n$}
We may assume that a $3$-dimensional subspace is generated by
\begin{align*}
\theta_1 &= \alpha_1 \nabla_1+ \alpha_2 \nabla_2+\alpha_3 \nabla_3 +\alpha_4 \nabla_4,\\
\theta_2 &= \beta_1 \nabla_1+ \beta_2 \nabla_2 +\beta_3 \nabla_3,\\
\theta_3 & = \gamma_1 \nabla_1+ \gamma_2 \nabla_2.
\end{align*}

Then we have the  following cases:

\; \;

    (a) $\alpha_4 \neq 0, \beta_3\neq 0, \gamma_2 \neq 0$.
\begin{enumerate}

    \item[(1)] $\alpha_1= 0, \beta_1= 0, \gamma_1 \neq 0, \alpha_2\neq \alpha_3$.
   It is easy to see that the orbit is
   $\langle \nabla_1 +\nabla_2, \alpha \nabla_2+ \nabla_3, \nabla_3+\nabla_4 \rangle$.

    \item[(2)] $\alpha_1= 0, \beta_1= 0, \gamma_1 \neq 0, \alpha_2=\alpha_3, \beta_2\neq \beta_3$.
   Then we have the case (a.1).

    \item[(3)] $\alpha_1= 0, \beta_1= 0, \gamma_1 \neq 0, \alpha_2=\alpha_3, \beta_2= \beta_3$.
   It is easy to see that the orbit is
   $\langle \nabla_1 +\nabla_2, \nabla_2+ \nabla_3, \nabla_4 \rangle$.

 \item[(4)] $\alpha_1= 0, \beta_1\neq 0, \gamma_1 = 0, \alpha_2 = \alpha_3$.
   It is easy to see that the orbit is
   $\langle \nabla_1 +\nabla_3, \nabla_2, \nabla_4 \rangle$.

    \item[(5)] $\alpha_1= 0, \beta_1\neq 0, \gamma_1 = 0, \alpha_2\neq \alpha_3$.
   It is easy to see that we have the case (a4).

    \item[(6)] $\alpha_1\neq 0, \beta_1= 0, \gamma_1 = 0$.
   It is easy to see that the orbit is
   $\langle \nabla_1 +\nabla_4, \nabla_2, \nabla_3 \rangle$.

    \end{enumerate}

    \;  \;

(b) $\alpha_4 \neq 0, \beta_3\neq 0, \gamma_2 = 0$.
\begin{enumerate}

    \item[(1)]    $\beta_2 \neq \beta_3$.
   It is easy to see that we can choose a representative with $\alpha_4 \neq0, \beta_3 \neq 0, \gamma_2\neq 0$
   and so we have the case (a).

    \item[(2)]
    $\alpha_1= 0, \beta_1= 0, \gamma_1 \neq 0, \alpha_2\neq \alpha_3, \beta_2=\beta_3$.
   It is easy to see that the orbit is
   $\langle \nabla_1,  \nabla_2+ \nabla_3, \nabla_3+\nabla_4 \rangle$.

    \item[(3)] $\alpha_1= 0, \beta_1= 0, \gamma_1 \neq 0, \alpha_2=\alpha_3, \beta_2 = \beta_3$.
   It is easy to see that the orbit is
   $\langle \nabla_1, \nabla_2+ \nabla_3, \nabla_4 \rangle$.

\end{enumerate}

    \;  \;

(c) $\alpha_4 \neq 0, \beta_3= 0, \gamma_2 = 0$.

\begin{enumerate}
    \item[(1)]   It is easy to see that we can choose a representative with $\alpha_4 \neq0, \beta_3 \neq 0$
   and so we have the cases (b) or (a).

    \end{enumerate}

 \;  \;

(d) $\alpha_4 =0, \beta_3= 0, \gamma_2 = 0$.

\quad   It is easy to see that the orbit is
   $\langle \nabla_1,  \nabla_2, \nabla_3 \rangle$.

It is easy to verify that all previous orbits  are different, and so we obtain
\begin{align*}
T_3(\mu_{1,1}^n) & =
\Big\langle \nabla_1,  \nabla_2 , \nabla_3 \Big\rangle \cup
\Big\langle \nabla_1, \nabla_2+ \nabla_3, \nabla_4 \Big\rangle \cup
\Big\langle \nabla_1, \nabla_2+ \nabla_3, \nabla_3+\nabla_4 \Big\rangle\\
&  \qquad  \cup \Big\langle \nabla_1 +\nabla_2, \nabla_2+ \nabla_3, \nabla_4 \Big\rangle \cup
\Big\langle \nabla_1 +\nabla_2, \alpha \nabla_2+ \nabla_3, \nabla_3+\nabla_4 \Big\rangle \\
 & \qquad  \cup
\Big\langle \nabla_1 +\nabla_3, \nabla_2, \nabla_4 \Big\rangle \cup
\Big\langle \nabla_1 +\nabla_4, \nabla_2, \nabla_3 \Big\rangle.
\end{align*}

\subsubsection{$4$-dimensional central extensions of $\mu_{1,1}^n$}
There is only one $4$-dimensional non-split central extension of the algebra $\mu_{1,1}^n$.
It is defined by $\langle \nabla_1, \nabla_2, \nabla_3, \nabla_4 \rangle$.

\subsubsection{Non-split central extensions of $\mu_{1,1}^n$}
So we have the next theorem

\begin{theorem}
An arbitrary non-split central extension of the algebra $\mu_{1,1}^n$ is isomorphic to one of the following pairwise non-isomorphic algebras

\begin{itemize}
\item one-dimensional central extensions:
\[\mu_{1,1}^{n+1}, \quad \mu_{1,2}^{n+1}, \quad \mu_{1,3}^{n+1}, \quad \mu_{1,4}^{n+1}, \quad \mu_{2,1}^{n+1}, \quad \mu_{2,2}^{n+1}(\alpha), \quad \mu_{2,3}^{n+1}, \quad \mu_{2,4}^{n+1}; \]

\item two-dimensional central extensions:
\[\mu_{2,1}^{n+2}, \! \quad \mu_{2,2}^{n+2}(\alpha),  \quad \mu_{2,3}^{n+2},   \quad \mu_{2,4}^{n+2}, \quad \mu_{2,5}^{n+2}, \quad \mu_{2,6}^{n+2}, \quad \mu_{2,7}^{n+2}, \quad \mu_{2,8}^{n+2}, \quad \mu_{2,9}^{n+2}(\alpha), \quad \mu_{2,10}^{n+2}, \]
\[ \mu_{3,1}^{n+2}, \quad \mu_{3,2}^{n+2}, \quad \mu_{3,3}^{n+2}(\alpha), \quad \mu_{3,4}^{n+2}; \]

\item three-dimensional central extensions:
\[\mu_{3,1}^{n+3}, \quad \mu_{3,2}^{n+3}, \quad \mu_{3,3}^{n+3}(\alpha), \quad \mu_{3,4}^{n+3}, \quad \mu_{3,5}^{n+3}, \quad \mu_{3,6}^{n+3}, \quad \mu_{3,7}^{n+3}; \]
\item four-dimensional central extensions:
\[\mu_{4,1}^{n+4}.\]
\end{itemize}
\end{theorem}

\subsection{Central extensions of $\mu_{1,2}^n$}

Let us denote
\[ \nabla_1= [\Delta_{1,n}], \ \ \nabla_2= [\Delta_{n,1}], \ \ \nabla_3= [\Delta_{n,n}]\in H^2(\mu_{1,2}^n,\mathbb{C}) \]
and $x=a_{1,1},y=a_{n-1,n},z=a_{n,1}$.  Let $\theta=\alpha_1 \nabla_1 +\alpha_2 \nabla_2+\alpha_3\nabla_3$. Then by
\[\left(
  \begin{array}{ccccc}
    \ast & \dots & 0 & 0 & \alpha_1^\prime \\
    0 & \dots & 0 & 0 & 0 \\
    \vdots & \ldots & \vdots & \vdots & \vdots \\
    0 & \dots & 0 & 0 & 0 \\
    \alpha_2^\prime & \dots & 0 & 0 & \alpha_3^\prime \\
  \end{array}
\right)=(\phi_{1,2}^n)^T\left(
  \begin{array}{ccccc}
    0 & \dots & 0 & 0 & \alpha_1 \\
    0 & \dots & 0 & 0 & 0 \\
    \vdots & \ldots & \vdots & \vdots & \vdots \\
    0 & \dots & 0 & 0 & 0 \\
    \alpha_2 & \dots & 0 & 0 & \alpha_3 \\
  \end{array}
\right)\phi_{1,2}^n,\]
we have the action of the automorphism group on the subspace $\langle \theta \rangle$ as
\[\langle
x^{(n-1)/2}(\alpha_1x+\alpha_3z)\nabla_1 + x^{(n-1)/2}(\alpha_2x+\alpha_3z)\nabla_2 + \alpha_3 x^{n-1} \nabla_3\rangle.\]

\subsubsection{$1$-dimensional central extensions of $\mu_{1,2}^n$}
Let us consider the following cases:

\begin{enumerate}

\item[(a)] $\alpha_3 =0$.

        \begin{enumerate}
        \item[(1)] $\alpha_1=0, \alpha_2 \neq 0$.
        Choosing $x=\frac{1}{\sqrt[\leftroot{-6}\uproot{3}n+1]{\alpha_2^2}}$, we have a representative $\nabla_2$ and the orbit is $\langle \nabla_2 \rangle$.

        \item[(2)] $\alpha_1\neq0$.
        Choosing $x= \frac{1}{\sqrt[\leftroot{-6}\uproot{3}n+1]{\alpha_1^2}}$ and $\alpha=\frac{\alpha_2}{\alpha_1}$, we have a representative $\nabla_1+\alpha \nabla_2$ and the orbit is $\langle \nabla_1+\alpha \nabla_2 \rangle$.

    \end{enumerate}

\item[(b)] $\alpha_3 \neq 0$.

    \begin{enumerate}
        \item[(1)] $\alpha_1=\alpha_2$.
        Choosing $x = \frac{1}{\sqrt[\leftroot{-3}\uproot{3}n-1]{\alpha_3}}, z= -\frac{\alpha_1}{\sqrt[\leftroot{-3}\uproot{3}n-1]{\alpha_3^n}}$, we have a representative $\nabla_3$ and the orbit is
        $\langle \nabla_3 \rangle$.

        \item[(2)] $\alpha_1\neq\alpha_2$.
        Choosing $x = \sqrt[\leftroot{-3}\uproot{3}n-3]{\frac{(\alpha_1-\alpha_2)^2}{\alpha_3^2}}, z= - \alpha_2\sqrt[\leftroot{-3}\uproot{3}n-3]{\frac{(\alpha_1-\alpha_2)^2}{\alpha_3^{n-1}}}$,
        we have a representative $\sqrt[\leftroot{-3}\uproot{3}n-3]{\frac{(\alpha_1-\alpha_2)^{2(n-1)}}{\alpha_3^{n+1}}}(\nabla_1+ \nabla_3)$ and the orbit is $\langle \nabla_1+ \nabla_3 \rangle$.
    \end{enumerate}
\end{enumerate}

It is easy to verify that all previous  orbits   are different, and so we obtain
\[
T_1(\mu_{1,2}^n)=
\Big \langle \nabla_1 +\alpha\nabla_2 \Big \rangle \cup
\Big \langle \nabla_1+ \nabla_3 \Big \rangle \cup
\Big \langle \nabla_2 \Big \rangle \cup
\Big \langle \nabla_3  \Big \rangle.
\]

\subsubsection{$2$-dimensional central extensions of $\mu_{1,2}^n$}

We may assume that a $2$-dimensional subspace is generated by
\begin{align*}
\theta_1 & = \alpha_1 \nabla_1+ \alpha_2 \nabla_2+\alpha_3 \nabla_3, \\
\theta_2 & = \beta_1 \nabla_1+ \beta_2 \nabla_2.
\end{align*}

Then we have the two following cases:

\

(a) $\alpha_3\neq0$.
\begin{enumerate}
    \item[(1)] $\alpha_1=\alpha_2, \beta_1\neq0$.
   Choosing $x = \frac{1}{\sqrt[\leftroot{-6}\uproot{3}n+1]{\beta_1^2}},
   z= - \frac{\alpha_1}{\alpha_3\sqrt[\leftroot{-6}\uproot{3}n+1]{\beta_1^2}}, \alpha=\frac{\beta_2}{\beta_1}$, we have a representative
   $\{\frac{\alpha_3}{\sqrt[\leftroot{-6}\uproot{3}n+1]{\beta_1^{2(n-1)}}} \nabla_3, \nabla_1+\alpha\nabla_2 \}$ and the orbit is $\langle \nabla_1+\alpha\nabla_2, \nabla_3  \rangle$.

    \item[(2)] $\alpha_1=\alpha_2, \beta_1=0, \beta_2\neq0$.
   Choosing $x = \frac{1}{\sqrt[\leftroot{-6}\uproot{3}n+1]{\beta_2^2}},
   z= - \frac{\alpha_1}{\alpha_3\sqrt[\leftroot{-6}\uproot{3}n+1]{\beta_2^2}}$, we have a representative
   $\{ \frac{\alpha_3}{\sqrt[\leftroot{-6}\uproot{3}n+1]{\beta_2^{2(n-1)}}}\nabla_3, \nabla_2 \}$ and the orbit is $\langle \nabla_2, \nabla_3 \rangle$.

   \item[(3)] $\alpha_1\neq\alpha_2, \beta_1\neq0, \beta_1\neq\beta_2$.
We can choose a representative with $\alpha_1=\alpha_2$, and so we have the case (a.1).

   \item[(4)] $\alpha_1\neq\alpha_2, \beta_1\neq0, \beta_1=\beta_2$.
   Choosing $x = \sqrt[\leftroot{-3}\uproot{3}n-3]{\frac{(\alpha_2-\alpha_1)^2}{\alpha_3^2}}, z= - \alpha_1\sqrt[\leftroot{-3}\uproot{3}n-3]{\frac{(\alpha_2-\alpha_1)^2}{\alpha_3^{n-1}}} $, we have a representative
   \[\bigg\{ \sqrt[\leftroot{-3}\uproot{3}n-3]{\frac{(\alpha_2-\alpha_1)^{2(n-1)}}{\alpha_3^{n+1}}}(\nabla_2+\nabla_3), \beta_1\sqrt[\leftroot{-3}\uproot{3}n-3]{\frac{(\alpha_2-\alpha_1)^{n+1}}{\alpha_3^{n+1}}}(\nabla_1+\nabla_2) \bigg\}\] and the orbit is $\langle  \nabla_1+\nabla_2,\nabla_2+\nabla_3\rangle$.

    \item[(5)] $\alpha_1\neq\alpha_2, \beta_1=0, \beta_2\neq0$.
We can choose a representative with $\alpha_1=\alpha_2$, and so we have the case (a.2).
\end{enumerate}

 \;  \;

(b) $\alpha_3=0$.

\quad  Then it is easy to see that the orbit is $\langle \nabla_1,\nabla_2\rangle$.

It is easy to verify that all previous orbits  are different, and so we obtain
\[
T_2(\mu_{1,2}^n)=
\Big \langle \nabla_1, \nabla_2 \Big \rangle \cup
\Big \langle \nabla_1+\nabla_2, \nabla_2+\nabla_3 \Big \rangle \cup
\Big \langle \nabla_1+\alpha\nabla_2, \nabla_3 \Big \rangle \cup
\Big \langle \nabla_2, \nabla_3 \Big \rangle.
\]

\subsubsection{$3$-dimensional central extensions of $\mu_{1,2}^n$}
There is only one $3$-dimensional non-split central extension of the algebra $\mu_{1,2}^n$.
It is defined by $\langle \nabla_1, \nabla_2, \nabla_3 \rangle$.

\subsubsection{Non-split central extensions of $\mu_{1,2}^n$}
So we have the next result.

\begin{theorem}
An arbitrary non-split central extension of the algebra $\mu_{1,2}^n$ is isomorphic to one of the following pairwise non-isomorphic algebras

\begin{itemize}
\item one-dimensional central extensions:
\[\mu_{2,3}^{n+1}, \quad \mu_{2,6}^{n+1}, \quad \mu_{2,9}^{n+1}(\alpha), \quad \mu_{2,10}^{n+1};\]

\item two-dimensional central extensions:
\[\mu_{3,3}^{n+2}(\alpha), \quad \mu_{3,4}^{n+2}, \quad \mu_{3,6}^{n+2}, \quad \mu_{3,7}^{n+2};\]
\item three-dimensional central extensions:
\[\mu_{4,2}^{n+3}.\]
\end{itemize}
\end{theorem}


\subsection{Central extensions of $\mu_{1,3}^n$}

We are using notations as in previous subsection.
Then by
\[\left(
  \begin{array}{ccccc}
    \ast & \dots & 0 & 0 & \alpha_1^\prime \\
    0 & \dots & 0 & 0 & 0 \\
    \vdots & \ldots & \vdots & \vdots & \vdots \\
    0 & \dots & 0 & 0 & 0 \\
    \alpha_2^\prime & \dots & 0 & 0 & \alpha_3^\prime \\
  \end{array}
\right)=(\phi_{1,3}^n)^T\left(
  \begin{array}{ccccc}
    0 & \dots & 0 & 0 & \alpha_1 \\
    0 & \dots & 0 & 0 & 0 \\
    \vdots & \ldots & \vdots & \vdots & \vdots \\
    0 & \dots & 0 & 0 & 0 \\
    \alpha_2 & \dots & 0 & 0 & \alpha_3 \\
  \end{array}
\right)\phi_{1,3}^n,\]
we have the action of the automorphism group on the subspace $\langle \theta \rangle$ as
\[\langle
(\alpha_1x^{n-1}+\alpha_3x^{n-2}z)\nabla_1 + (\alpha_2x^{n-1}+\alpha_3x^{n-2}z)\nabla_2 + \alpha_3 x^{2n-4} \nabla_3\rangle.\]

\subsubsection{$1$-dimensional central extensions of $\mu_{1,3}^n$}

We consider the following cases:

\begin{enumerate}
\item[(a)] $\alpha_3 =0$.
    \begin{enumerate}
        \item[(1)] $\alpha_1=0, \alpha_2 \neq 0$.
        Choosing $x=\frac{1}{\sqrt[\leftroot{-3}\uproot{3}n-1]{\alpha_2}}$, we have a representative $\nabla_2$ and the orbit is $\langle \nabla_2 \rangle$.

        \item[(2)] $\alpha_1\neq0$.
        Choosing $x= \frac{1}{\sqrt[\leftroot{-3}\uproot{3}n-1]{\alpha_1}}$ and $\alpha= \frac{\alpha_2}{\alpha_1}$, we have a representative $\nabla_1+\alpha \nabla_2$ and the orbit is $\langle \nabla_1+\alpha \nabla_2 \rangle$.

    \end{enumerate}

\item[(b)] $\alpha_3 \neq 0$.

\begin{enumerate}
    \item[(1)] $\alpha_1=\alpha_2$.
        Choosing $x = \frac{1}{\sqrt[\leftroot{-2}\uproot{2}2n-4]{\alpha_3}}, z= - \frac{\alpha_1}{\sqrt[\leftroot{-5}\uproot{3}n-2]{\alpha_3^{n-1}}}$, we have a representative $\nabla_3$ and the orbit is
        $\langle \nabla_3 \rangle$.

        \item[(2)] $\alpha_1\neq\alpha_2$.
        Choosing $x = \sqrt[\leftroot{-3}\uproot{3}n-3]{\frac{\alpha_1-\alpha_2}{\alpha_3}}, z= - \alpha_2\sqrt[\leftroot{-3}\uproot{3}n-3]{\frac{\alpha_1-\alpha_2}{\alpha_3^{n-2}}}$,
        we have a representative $\sqrt[\leftroot{-3}\uproot{3}n-3]{\frac{(\alpha_1-\alpha_2)^{2(n-2)}}{\alpha_3^{n-1}}}(\nabla_1+ \nabla_3)$ and the orbit is $\langle \nabla_1+ \nabla_3 \rangle$.
    \end{enumerate}
\end{enumerate}

It is easy to verify that all  previous orbits are different, and so we obtain
\[
T_1(\mu_{1,3}^n)=
\Big \langle \nabla_1 +\alpha\nabla_2 \Big \rangle \cup
\Big \langle \nabla_1+ \nabla_3 \Big \rangle \cup
\Big \langle \nabla_2 \Big \rangle \cup
\Big \langle \nabla_3  \Big \rangle .
\]

\subsubsection{$2$-dimensional central extensions of $\mu_{1,3}^n$}

We may assume that a $2$-dimensional subspace is generated by
\begin{align*}
\theta_1 & =\alpha_1 \nabla_1+ \alpha_2 \nabla_2+\alpha_3 \nabla_3, \\
\theta_2 & = \beta_1 \nabla_1+ \beta_2 \nabla_2.
\end{align*}

Then we have the two following cases:

\begin{enumerate}

\item[(a)] $\alpha_3\neq0$.
\begin{enumerate}
    \item[(1)] $\alpha_1=\alpha_2, \beta_1\neq0$.
   It is easy to see that the orbit is $\langle \nabla_1+\alpha\nabla_2, \nabla_3  \rangle$.

    \item[(2)] $\alpha_1=\alpha_2, \beta_1=0, \beta_2\neq0$.
   It is easy to see that the orbit is  $\langle \nabla_2, \nabla_3 \rangle$.

   \item[(3)] $\alpha_1\neq\alpha_2, \beta_1\neq0, \beta_1\neq\beta_2$.
It is easy to see that we have the case (a.1).

   \item[(4)] $\alpha_1\neq\alpha_2, \beta_1\neq0, \beta_1=\beta_2$.
  It is easy to see that the orbit is  $\langle  \nabla_1+\nabla_2,\nabla_2+\nabla_3\rangle$.

    \item[(5)] $\alpha_1\neq\alpha_2, \beta_1=0, \beta_2\neq0$.
   It is easy to see that we have the case (a.2).
\end{enumerate}

\item[(b)] $\alpha_3=0$.

Then it is easy to see that the orbit is $\langle \nabla_1,\nabla_2\rangle$.
\end{enumerate}

It is easy to verify that all previous orbits are different, and so we obtain
\[
T_2(\mu_{1,3}^n)=
\Big \langle \nabla_1, \nabla_2 \Big \rangle \cup
\Big \langle \nabla_1+\nabla_2, \nabla_1+\nabla_3 \Big \rangle \cup
\Big \langle \nabla_1+\alpha\nabla_2, \nabla_3 \Big \rangle \cup
\Big \langle \nabla_2, \nabla_3 \Big \rangle.
\]

\subsubsection{$3$-dimensional central extensions of $\mu_{1,3}^n$}
There is only one $3$-dimensional non-split central extension of the algebra $\mu_{1,3}^n$.
It is defined by $\langle \nabla_1, \nabla_2, \nabla_3 \rangle$.

\subsubsection{Non-split central extensions of $\mu_{1,3}^n$}
So we have the next theorem.

\begin{theorem}
An arbitrary non-split central extension of the algebra $\mu_{1,3}^n$ is isomorphic to one of the following pairwise non-isomorphic algebras

\begin{itemize}
\item one-dimensional central extensions:
\[\mu_{2,1}^{n+1}, \quad \mu_{2,2}^{n+1}(\alpha)_{\alpha\neq1}, \quad \mu_{2,5}^{n+1}, \quad \mu_{2,6}^{n+1}, \quad \mu_{2,7}^{n+1};\]
\item two-dimensional central extensions:
\[\mu_{3,1}^{n+2}, \quad \mu_{3,3}^{n+2}(\alpha)_{\alpha\neq1}, \quad \mu_{3,4}^{n+2}, \quad \mu_{3,5}^{n+2}, \quad \mu_{3,6}^{n+2};\]
\item three-dimensional central extensions:
\[\mu_{4,3}^{n+3}.\]
\end{itemize}
\end{theorem}

\subsection{Central extensions of $\mu_{1,4}^n$}

We are using notations as in previous subsection. Analogously to the previous cases by
\[\left(
  \begin{array}{ccccc}
    \ast & \dots & 0 & 0 & \alpha_1^\prime \\
    0 & \dots & 0 & 0 & 0 \\
    \vdots & \ldots & \vdots & \vdots & \vdots \\
    0 & \dots & 0 & 0 & 0 \\
    \alpha_2^\prime & \dots & 0 & 0 & \alpha_3^\prime \\
  \end{array}
\right)=(\phi_{1,4}^n)^T\left(
  \begin{array}{ccccc}
    0 & \dots & 0 & 0 & \alpha_1 \\
    0 & \dots & 0 & 0 & 0 \\
    \vdots & \ldots & \vdots & \vdots & \vdots \\
    0 & \dots & 0 & 0 & 0 \\
    \alpha_2 & \dots & 0 & 0 & \alpha_3 \\
  \end{array}
\right)\phi_{1,4}^n,\]
for any
$\theta=\alpha_1 \nabla_1 +\alpha_2 \nabla_2 +\alpha_3 \nabla_3$, we have
the action of the automorphism group on the subspace $\langle \theta \rangle$ as
\[\langle
(\alpha_1+\alpha_3z)\nabla_1 + (\alpha_2+\alpha_3z)\nabla_2 + \alpha_3 \nabla_3\rangle.\]

\subsubsection{$1$-dimensional central extensions of $\mu_{1,4}^n$}
Let us consider the following cases:

\begin{enumerate}
\item[(a)] $\alpha_3 =0$.
    \begin{enumerate}
        \item[(1)] $\alpha_1=0, \alpha_2 \neq 0$.
        Then it is easy to see that the orbit is $\langle \nabla_2 \rangle$.

        \item[(2)] $\alpha_1\neq0$.
        Then it is easy to see that the orbit is
$\langle \nabla_1+\alpha \nabla_2 \rangle$.

    \end{enumerate}

\item[(b)] $\alpha_3 \neq 0$.

Then it is easy to see that the orbit is $\langle \alpha\nabla_1+\nabla_3 \rangle$.

\end{enumerate}

It is easy to verify that all previous orbits are different, and so we obtain
\[
T_1(\mu_{1,4}^n)=
\Big \langle \nabla_1 +\alpha\nabla_2 \Big \rangle \cup
\Big \langle \alpha\nabla_1+ \nabla_3 \Big \rangle \cup
\Big \langle \nabla_2 \Big \rangle.
\]

\subsubsection{$2$-dimensional central extensions of $\mu_{1,4}^n$}
We may assume that a $2$-dimensional subspace is generated by
\begin{align*}
\theta_1 & = \alpha_1 \nabla_1+ \alpha_2 \nabla_2+\alpha_3 \nabla_3,  \\
\theta_2 & = \beta_1 \nabla_1+ \beta_2 \nabla_2.
\end{align*}







Then we have the three following cases:

\begin{enumerate}

   \item[(a)]   $\alpha_3 \neq 0,  \beta_1\neq 0, \alpha_1\neq\alpha_2, \beta_1=\beta_2$.
  Choosing $z= - \frac{\alpha_2}{\alpha_3}, \alpha=\frac{\alpha_1-\alpha_2}{\alpha_3}\neq0$,
    we have a representative $\{ \alpha_3(\alpha\nabla_1+ \nabla_3), \beta_1(\nabla_1+\nabla_2)\}$ and the orbit is $\langle \nabla_1+\nabla_2, \alpha\nabla_1+ \nabla_3\rangle_{\alpha\neq0}$.

    \item[(b)]  $\alpha_3 \neq 0,  \beta_1\neq 0,  \beta_1\neq\beta_2$.
  Choosing $z=\frac{\alpha_1\beta_2-\alpha_2\beta_1}{(\beta_1-\beta_2)\alpha_3}, \alpha=\frac{\beta_2}{\beta_1}$,
    we have a representative   \[\Big\{ \frac{\alpha_1-\alpha_2}{\beta_1-\beta_2}\beta_1(\nabla_1+ \alpha\nabla_2)+\alpha_3\nabla_3, \beta_1(\nabla_1+\alpha\nabla_2)\Big\}\] and the orbit is $\langle \nabla_1+\alpha\nabla_2, \nabla_3\rangle$.


    \item[(d)]   $\alpha_3 \neq 0,  \beta_1= 0, \beta_2 \neq0$.
    It is easy to see that the orbit is $\langle \nabla_2, \nabla_3 \rangle$.

    \item[(e)]   $\alpha_3 = 0$.
   It is easy to see that the orbit is $\langle \nabla_1,  \nabla_2 \rangle$.

\end{enumerate}

It is easy to verify that all previous orbits  are different, and so we obtain
\[
T_2(\mu_{1,4}^n)=
\Big \langle  \nabla_1,  \nabla_2 \Big \rangle \cup
\Big \langle \nabla_1+\alpha\nabla_2, \nabla_3 \Big \rangle \cup
\Big \langle \nabla_1+\nabla_2, \alpha\nabla_1+ \nabla_3 \Big\rangle_{\alpha\neq0} \cup
\Big \langle  \nabla_2, \nabla_3 \Big \rangle.
\]

\subsubsection{$3$-dimensional central extensions of $\mu_{1,4}^n$}
There is only one non-split $3$-dimensional central extension of $\mu_{1,4}^n$.
It is defined by $\langle \nabla_1, \nabla_2, \nabla_3 \rangle$.

So we have the following result.

\begin{theorem}
An arbitrary non-split central extension of the algebra $\mu_{1,4}^n$ is isomorphic to one of the following pairwise non-isomorphic algebras

\begin{itemize}
\item one-dimensional central extensions:
\[\mu_{2,2}^{n+1}(1), \quad \mu_{2,7}^{n+1}, \quad \mu_{2,8}^{n+1}, \quad \mu_{2,9}^{n+1}(\alpha)_{\alpha\neq1}, \quad \mu_{2,10}^{n+1};\]
\item two-dimensional central extensions:
\[\mu_{3,2}^{n+2}, \quad \mu_{3,3}^{n+2}(\alpha)_{\alpha\neq1}, \quad \mu_{3,4}^{n+2}, \quad \mu_{3,5}^{n+2}, \quad \mu_{3,6}^{n+2}, \quad \mu_{3,7}^{n+2};\]
\item three-dimensional central extensions:
\[\mu_{4,4}^{n+3}.\]

\end{itemize}
\end{theorem}

\section*{Appendix: The list of main algebras}

\[\begin{array}{lllll}
\mu_{2,5}^n: & e_ie_j=e_{i+j},& e_1e_{n-1}=e_{n-2}+e_n, & e_{n-1}e_1=e_n,  \\
\mu_{2,6}^n: & e_ie_j=e_{i+j},& e_1e_{n-1}=e_{n-2}, & e_{n-1}e_1=e_n, & e_{n-1}e_{n-1}=e_n, \\
\mu_{2,7}^n: & e_ie_j=e_{i+j},& e_1e_{n-1}=e_{n-2}, & e_{n-1}e_{n-1}=e_n, \\
\mu_{2,8}^n: & e_ie_j=e_{i+j},& e_1e_{n-1}=e_{n-2}+e_n, &e_{n-1}e_1=e_n, & e_{n-1}e_{n-1}=e_{n-2}, \\
\mu_{2,9}^n(\alpha): & e_ie_j=e_{i+j}, & e_1e_{n-1}=e_n, & e_{n-1}e_1=\alpha e_n, & e_{n-1}e_{n-1}=e_{n-2}, \\
\mu_{2,10}^n: & e_ie_j=e_{i+j},& e_{n-1}e_1=e_n, & e_{n-1}e_{n-1}=e_{n-2}, \\
\multicolumn{5}{l}{\text{where $2\leq i+j\leq n-2$ \ and  \ $\alpha\in\mathbb{C}$.}}\\ \\

\mu_{3,1}^n: & e_ie_j=e_{i+j},& e_1e_{n-2}=e_{n-1}, & e_{n-2}e_1=e_n, \\
\mu_{3,2}^n: & e_ie_j=e_{i+j},& e_1e_{n-2}=e_{n-1}, & e_{n-2}e_1=e_{n-1}+e_n, & e_{n-2}e_{n-2}=e_n  \\
\mu_{3,3}^n(\alpha): & e_ie_j=e_{i+j},& e_1e_{n-2}=e_{n-1}, &e_{n-2}e_1=\alpha e_{n-1}, & e_{n-2}e_{n-2}=e_n, \\
\mu_{3,4}^n: & e_ie_j=e_{i+j},&  e_{n-2}e_1=e_{n-1}, & e_{n-2}e_{n-2}=e_n\\
\mu_{3,5}^n: & e_ie_j=e_{i+j},&  e_1e_{n-2}=e_{n-3}+e_{n-1}, & e_{n-2}e_1=e_{n-1}, & e_{n-2}e_{n-2}=e_n,  \\
\mu_{3,6}^n: & e_ie_j=e_{i+j},&  e_1e_{n-2}=e_{n-3}+e_{n-1}, & e_{n-2}e_1=e_{n-1}+e_n, & e_{n-2}e_{n-2}=e_n,  \\
\mu_{3,7}^n: & e_ie_j=e_{i+j},&  e_1e_{n-2}=e_{n-1}, & e_{n-2}e_1=e_n, & e_{n-2}e_{n-2}=e_{n-3},\\
\multicolumn{5}{l}{\text{where $2\leq i+j\leq n-3$ \ and  \ $\alpha\in\mathbb{C}$.}}\\ \\

\mu_{4,1}^n: & e_ie_j=e_{i+j},& e_1e_{n-3}=e_{n-2}, & e_{n-3}e_1=e_{n-1}, & e_{n-3}e_{n-3}=e_n, \\
\mu_{4,2}^n: & e_ie_j=e_{i+j},& e_1e_{n-3}=e_{n-2}, & e_{n-3}e_1=e_{n-1}, & e_{n-3}e_{n-3}=e_{n-4}+e_n, \\
\mu_{4,3}^n: & e_ie_j=e_{i+j},& e_1e_{n-3}=e_{n-4}+e_{n-2}, & e_{n-3}e_1=e_{n-1}, & e_{n-3}e_{n-3}=e_n, \\
\mu_{4,4}^n: & e_ie_j=e_{i+j},& e_1e_{n-3}=e_{n-4}+e_{n-2}, & e_{n-3}e_1=e_{n-1}, & e_{n-3}e_{n-3}=e_{n-4}+e_n,\\
\multicolumn{5}{l}{\text{where $2\leq i+j\leq n-4$ \ and  \ $\alpha\in\mathbb{C}$.}}
\end{array}
\]


\end{document}